\documentclass[11pt,reqno,a4paper]{amsart}

\usepackage{amssymb,epsfig,graphics}

\newtheorem{theorem}{Theorem}[section]
\newtheorem{proposition}[theorem]{Proposition}

\newtheorem{lemma}[theorem]{Lemma}
\newtheorem{sub-lemma}[theorem]{Sub-Lemma}
\newtheorem{claim}[theorem]{Claim}

\def\C{\mathcal{C}}
\def\Z{\mathcal{Z}}

\def\M{\mathcal{M}}
\def\R{\mathcal{R}}

\def\D{\Delta}
\def\NN{\mathbb{N}}
\def\PP{\mathbb{P}}

\def\TT{\mathbb{T}}

\def\DD{\mathbb{D}}
\def\CC{\mathbb{C}}

\DeclareMathOperator{\supp}{supp}
\DeclareMathOperator{\diam}{diam}
\DeclareMathOperator{\dist}{dist}

\def\M{\mathcal{M}}

\let\eps=\varepsilon

\def\D{\mathcal{D}}

\def\Z{\mathcal{Z}}

\def\1{{{\mathit 1} \!\!\>\!\! I} }
\renewcommand{\liminf}{\mathop{{\underline {\hbox{{\rm lim}}}}}}
\renewcommand{\limsup}{\mathop{{\overline {\hbox{{\rm lim}}}}}}

\def\rec{{\rm Rec}}

\DeclareMathOperator{\Leb}{Leb}
\DeclareMathOperator{\Cov}{Cov}

\begin{document}

\title[Poisson law for hyperbolic systems with polynomial decay]{Poisson law for some nonuniformly hyperbolic dynamical systems with polynomial rate of mixing}
\author{Fran\c{c}oise P\`ene \and Beno\^\i t Saussol}
\address{1)Universit\'e Europ\'eenne de Bretagne, 
France\\
2)Universit\'e de Bretagne Occidentale, Laboratoire de
Math\'ematiques de Bretagne Atlantique, CNRS UMR 6205, Brest, France\\
3)partially supported
by the ANR project PERTURBATIONS (ANR-10-BLAN 0106)}
\email{francoise.pene@univ-brest.fr}
\email{benoit.saussol@univ-brest.fr}
%\urladdr{}
\begin{abstract}
We consider some nonuniformly hyperbolic invertible dynamical systems which are modeled by a Gibbs-Markov-Young tower.
We assume a polynomial tail for the inducing time and a polynomial control of hyperbolicity, as introduced by Alves, Pinheiro and Azevedo. These systems admit a physical measure with polynomial rate of mixing.
In this paper we prove that the distribution of the number of visits to a ball $B(x,r)$ 
converges to a Poisson distribution as the radius $r\to0$ and after suitable normalization.
\end{abstract}
\date{\today}
\maketitle
\bibliographystyle{plain}
\tableofcontents
\section{Introduction}
Many dynamical systems with some hyperbolicity enjoy strong statistical properties. Let us mention a few of them:
existence of physical measure, exponential decay of correlations, central limit theorem, large  deviation principles, etc.
That is, the probabilistic behavior of these systems mimics an i.i.d. process.
Beyond uniform hyperbolicity the situation may be different.
This has a visible consequence for example in the validity of the CLT, 
which is often related to a summable decay of correlation. 

We will consider here a class of nonuniformly hyperbolic systems for which the 
polynomial decay of correlation can be arbitrarily slow. The setting, 
introduced by Alves and Pinheiro~\cite{AlvesPinheiro} and generalized by Alves 
and Azevedo~\cite{aa} is a system modeled by a Young tower on which we have a 
uniform polynomial control on the
contraction along stable manifolds and backward contraction along unstable 
manifolds. The full description of the setting is in Section~\ref{sec:tower}.

Let $f$ be an invertible map defined on a riemaniann manifold $\M$ giving rise 
to a metric $d$. Suppose that $\mu$ is an invariant measure for $f$. Given $x$ in
$\M$, we are 
interested in the statistical behavior, with respect to $\mu$, of the number of 
occurences of entrance times in the ball $B(x,r)$. Namely, setting
\[
N(x,r)(y)= \sharp \left\{n\in\NN\colon d(f^ny,x)<r,\, 1\le n\le \frac{t}{\mu(B(x,r))}\right\}
\]
we are interested in the limit distribution of $N(x,r)$, as $r$ goes to $0$. 
The main result of the paper is that this limit is the Poisson distribution, for $\mu$-a.e. $x$. This question has been adressed for many different dynamical systems.
Our result is a generalization of the recent work by Collet and Chazottes~\cite{cc} who studied towers with exponential tail of return time to a setting where the tail is only polynomial (We refer to \cite{cc} and references therein for  details on previous works). Our work applies for example to stadium like billiards.
During the preparation of this work Freitas, Haydn and Nicol~\cite{fhn} obtained 
a similar result on Poisson distribution for these billiards by a different method (inducing the billiard map on a suitable reference set where the induced map has a tower with exponential tail). We also mention Haydn and Wasilewska \cite{wasilewska} whose approach needs a polynomial tail of sufficiently large order, in a nonuniformly expanding setting.

Our limit theorem relies on precise mixing estimates for sets defined with balls. For some systems the indicator functions of balls have a bounded norm in a good Banach space of functions for which one can use the decay of correlations directly. This leads to a Poisson distribution and even gives stronger results~\cite{afv}.
Here in our situation one cannot expect to implement this strategy and we need to approximate our balls to get 
mixing estimates. That is why we need a control on the measure of 
neighborhood of balls as in~\eqref{couronne}. 
Outside absolutely continuous measure this leads to delicate questions. In~\cite{cc} a general result 
is obtained for SRB measure with one-dimensional unstable manifold. The generalization to higher dimensional systems 
or polynomial tower being open, we left the condition~\eqref{couronne} as an assumption. We emphasize that this condition
is the weakest that one can ask in our setting.

A major step in our proof of the Poisson distribution is to bypass the lengthy and delicate study of short
returns by a simple argument based on recurrence rates. Indeed we show that for our systems 
\[
\min\{n\ge1\colon d(f^nx,x)<r\} \approx r^{-\dim_H\mu}
\]
for $\mu$-a.e. $x$, where $\dim_H\mu$ stands for the Hausdorff dimension of $\mu$ 
(See Section \ref{sec:recrate} for precise statement). There are some systems with polynomial decay of correlation for which the above behavior does not hold and it turns out that the Poisson distribution cannot happen in these cases (See \cite{grs}).

\section{Assumptions}\label{sec:tower}
We consider an invertible transformation $f:\M\rightarrow \M$ on a finite-dimensional Riemannian manifold satisfying assumptions
of \cite{aa}. 
These assumptions ensure the existence of a $f$-invariant SRB probability measure $\mu$ and that $(\M,f,\mu)$ can be modeled by a 
Gibbs Markov Young tower with good properties as described briefly below for completeness.

Recall that a stable (resp. unstable) manifold is an embedded
disk $\gamma\subset \M$ such that, for every $x,y\in\gamma$, 
$\dist(f^{n}x,f^{n}y)\rightarrow 0$ 
(resp. $\dist(f^{-n}x,f^{-n}y)\rightarrow 0$) as $n$
goes to infinity.

Consider two continuous families $\Gamma^u$ and $\Gamma^s$ 
of respectively unstable and stable $C^1$ manifolds such that
there exists $\alpha_{\min}>0$ so that,
for every $(\gamma^s,\gamma^u)\in\Gamma^s\times\Gamma^u$, we have
$$\dim\gamma^s+\dim\gamma^u=\dim \M,\ \ \#(\gamma^s\cap\gamma^u)=1\ \ \mbox{and}\ \ |\angle(\gamma^s,\gamma^u)|\ge \alpha_{\min}.$$
We set $\Lambda:=(\bigcup_{\gamma^u\in\Gamma^u}\gamma^u)\cap(\bigcup_{\gamma^s\in\Gamma^s}\gamma^s)$ and assume that the following properties hold

\begin{itemize}
\item[(P1)] \underline{Markov}: there exists a family $(\Lambda_i)_{i\ge 1}$ of
pairwise disjoint subsets of the form $\Lambda_i=(\bigcup_{\gamma^u\in\Gamma^u}\gamma^u)\cap(\bigcup_{\gamma^s\in\Gamma_i^s}\gamma^s)$ for some family
$(\Gamma_i^s)_i$ of pairwise disjoint subsets of $\Gamma^s$ such that
\begin{itemize}
\item[(a)] for some $\gamma\in\Gamma^u$, we have $\Leb_\gamma(\Lambda)>0$ and 
$\Leb_\gamma(\Lambda\setminus\bigcup_i\Lambda_i)=0$;
\item[(b)] For every $i\ge 1$, there exists an integer $R_i\ge 1$
such that $f^{R_i}(\Lambda_i)=(\bigcup_{\gamma^u\in\Gamma_i^u}\gamma^u)\cap(\bigcup_{\gamma^s\in\Gamma^s}\gamma^s)$ for some $\Gamma_i^u\subset\Gamma^u$.
Moreover, for every $\gamma^s\in\Gamma_i^s$, there exists $\gamma_0^s\in\Gamma^s$ such that $f^{R_i}(\gamma^s)\subset\gamma_0^s$ and,
for every $\gamma^u\in\Gamma^u$, 
there exists $\gamma_0^u\in\Gamma_i^u$ such that $\gamma_0^u\subset f^{R_i}(\gamma^u)$.
\end{itemize}
\end{itemize}
This enables us to define a particular return time $R:\Lambda\rightarrow\mathbb N$ and the associated return map $f^R:\Lambda\rightarrow\Lambda$ by setting
$$R_{|\Lambda_i}\equiv R_i\ \ \mbox{and}\ \ (f^R)_{|\Lambda_i}\equiv f^{R_i}.$$
We also define a separation time $s:\Lambda\times\Lambda\rightarrow\mathbb N\cup\{\infty\}$ for the return map as follows:
$$\forall x,y\in\Lambda,\ \ s(x,y):=\min\{n\ge 0:\exists j\ge 1,\ 
(f^R)^n(x)\in\Lambda_j\ \mbox{and}\ (f^R)^n(y)\not\in\Lambda_j\}.$$
With these notations, we assume that 
there exist $\alpha>0$,
$\beta\in(0,1)$ and $C>0$ such that, for every $\gamma_0^u,\gamma_1^u\in\Gamma^u$ and every $\gamma^s\in\Gamma^s$, we have
\begin{itemize}
\item[(P2)] \underline{Polynomial contraction on stable leaves}: for every $x,y\in\gamma^s$
and every $n\ge 1$, we have $\dist(f^n(x),f^n(y))\le C n^{-\alpha}$; 
\item[(P3)] \underline{Backward polynomial contraction on unstable leaves}: for every $x,y\in\gamma_0^u$
and every $n\ge 1$, we have $\dist(f^{-n}(x),f^{-n}(y))\le C n^{-\alpha}$; 
\item[(P4)]\underline{Bounded distortion}: for every $x,y\in\gamma_0^u\cap\Lambda$, we have
$$\log\frac{\det D(f^R)_u(x)}{\det D(f^R)_u(y)}\le C\beta^{s(f^R(x),f^R(y))} .$$
\item[(P5)] \underline{Regularity of the stable foliation}: consider the map
$\Theta_{\gamma_0^u,\gamma_1^u}:
\gamma_0^u\cap\Lambda\rightarrow\gamma_1^u\cap \Lambda$ defined by
$\Theta_{\gamma_0^u,\gamma_1^u}(x)$ is the unique $x'$ for which
there exists $\gamma\in\Gamma^s$ such that $x\in\gamma\cap\gamma_0^u$
and $x'\in\gamma\cap\gamma_1^u$. We assume that
\begin{itemize}
\item[(a)] $\Theta_{\gamma_0^u,\gamma_1^u}$ is absolutely continuous and
$$U:=\frac{d((\Theta_{\gamma_0^u,\gamma_1^u})_*\Leb_{\gamma_0^u})}
  {d\Leb_{\gamma_1^u}}=\prod_{i\ge 0}\frac{\det Df_u \circ f^i}{\det
     Df_u\circ f^i\circ \Theta_{\gamma_0^u,\gamma_1^u}^{-1}};$$
\item[(b)]
for every $x,y\in\gamma_1^u$, we have $\log(U(x)/U(y))\le
    C\beta^{s(x,y)}$. 
\end{itemize}

\end{itemize}
We assume that $\gcd(R_i,i\ge 1)=1$.
We consider here the case when the return time $R$ has a polynomial tail
distribution, more precisely we assume that 
\begin{equation}\label{poly}
\Leb_\gamma(R>n)\le C n^{-\zeta},\ \ \mbox{for some}\ \zeta>1,
\end{equation}
which ensures the integrability of $R$ with respect to $\Leb_\gamma$.
Under these conditions, the systems admits a SRB measure $\mu$: $\mu$ has no zero Lyapunov 
exponent and the conditional measures on local unstable manifolds are absolutely continuous
with respect to the Lebesgue measure on these manifolds.

We recall the definition of the Hausdorff dimension of the measure $\mu$ 
as 
\[
\dim_H\mu = \inf_{\mu(Y)=1} \dim_H Y. 
\]
We make the standing assumption that
\begin{equation}\label{eq:deltak}
\alpha>\frac{1}{\dim_H\mu}.
\end{equation}
We say that  $u\precsim v$ to express that there exists a constant $C>0$ (depending only
on the dynamical system) such that $u\le C v$. We also write $u \approx v$ when $u\precsim v$ and $v\precsim u$.
\section{Poisson law for the number of entrance to balls}

For any $x\in \M$ and $r>0$, we write $B(x,r)$ for the 
ball of center $x$ and radius $r$ for the Riemannian distance on $\M$.
The main result of the paper states that for typical centers $x$, 
the time spent into the ball $B(x,r)$, up to time $t/\mu(B(x,r))$, follows asymptotically the Poisson law with mean $t$. 

We assume that, for $\mu$-almost every $x\in \M$, there exists $\delta\in(1,\alpha \dim_H\mu)$ such that
\begin{equation}\label{couronne}
\mu(B(x,r+r^\delta)\setminus B(x,r))= o(\mu(B(x,r)))
\end{equation}
(see appendix for discussion on (\ref{couronne})).
\begin{theorem}\label{THM}
Let $(\M,f,\mu)$ be as above. 
For $\mu$-a.e. $x\in \M$ such that \eqref{couronne} holds, 
\begin{equation}\label{eq:cvpoisson}
\lim_{r\to0} \mu \left(\left\{ y\in \M\colon \sum_{j=0}^{\lfloor t /\mu(B(x,r))\rfloor}
\mathbf{1}_{B(x,r)}(f^jy)=k\right\}\right) = \frac{t^k}{k!}e^{-t}
\end{equation}
for any $k\ge0$ and any $t\ge0$.
\end{theorem}

The condition \eqref{couronne} on the coronas is always satisfied for a subsequence:
(see the appendix for details).
\begin{proposition}\label{pro:subseq}
Let $\theta<1$. For every $x\in\mathcal M$, there exists a sequence $r_n(x)\in(\theta^{n+1},\theta^{n})$ such that for any $\delta\in(1,\alpha \dim_H\mu)$, \eqref{couronne} holds for the sequence $r=r_n(x)$.
\end{proposition}
A quick look at the proof of Theorem \ref{THM} (especially Proposition~\ref{propR1}) shows that the convergence \eqref{eq:cvpoisson} along the subsequence $r_n(x)$ still holds in this case, provided the other assumptions are satisfied.

\subsection{An abstract Poisson approximation result}

We use the method developed by Chazottes and Collet, Theorem 2.1 in \cite{cc}.
We recall that, for any probability measures $P$ and $Q$ on a same measurable space $(E,\mathcal E)$, the total variation distance
$d_{TV}(P,Q)$ between $P$ and $Q$ is
$$d_{TV}(P,Q)=\sup_{A\in\mathcal E}P(A)-Q(A). $$
If $Y$ and $Z$ are random variables taking integer values, the total variation distance between their law 
is written $d_{TV}(Y,Z)$ (with a small abuse of notation) and is given by 
\[
d_{TV}(Y,Z) = \frac12 \sum_{k=0}^\infty |\PP(Y=k)-\PP(Z=k)|.
\]
By $\mathrm{Poisson}(\lambda)$  we denote a Poisson random variable with mean $\lambda>0$, namely
\[
\PP(\mathrm{Poisson}(\lambda)=k) = \frac{\lambda^k}{k!}e^{-k}.
\]

\begin{theorem}[\cite{cc}]\label{thm:poisson}
Let $(X_n)_{n\in\NN}$ be a stationary $\{0,1\}$-valued process and $\eps:=\PP(X_1=1)$.
Set $S_i^j=X_i+\cdots+X_j$ for any $0\le i\le j$.
Then, for all positive integers $p,M,N$ such that $M\le N-1$, and $2\le p<N$, one has
\[
d_{TV}(X_1+\cdots+X_N,\mathrm{Poisson}(N\eps)) \le R(\eps,N,p,M)
\]
with
\[
R(\eps,N,p,M) = 2NM[ R_1(\eps,N,p)+R_2(\eps,p)]+R_3(\eps,N,p,M)
\]
where
\[
\begin{split}
R_1(\eps,N,p) &:= \sup_{{}^{0\le j\le N-p}_{0\le q\le N-j-p}}\left|\PP({X_1=1}, {S_{p+1}^{N-j}=q})-\eps\PP({S_{p+1}^{N-j}=q})\right| \\
R_2(\eps,p) &:= \PP(X_1=1, S_{2}^{p}\ge 1) \\
R_3(\eps,N,p,M) &:= 4\left( Mp\eps(1+N\eps)+\frac{(\eps N)^M}{M!}e^{-N\eps}+N\eps^2\right).
\end{split}
\]
\end{theorem}
The definition of $R_2$ given here differs slightly from the original statement, but
an eye to the proof shows that the arguments go trough with this tiny change.
However, this modification plays an essential role in the present paper since it allows a 
very efficient estimate of the term $R_2$.

\subsection{Proof of the main theorem}

For a fixed $x\in \M$, we will apply Theorem~\ref{thm:poisson} to the processes defined for $y\in \M$ by 
$X_{n+1}(y)=\mathbf{1}_{B(x,r)}(f^{n}y)$, $n\in\NN$, $r>0$. 
For any integers $p\le q$, we write $S_p^q(x,r)$ for the number of visits to $B(x,r)$ of the orbit of $f$ 
between times $p$ and $q$, i.e.
$$\forall y\in \M,\ \ S_p^q(x,r)(y):=\#\{\ell=p,...,q\ :\ f^\ell(y)\in B(x,r)\}. $$

We denote the error terms $R_1$ and $R_2$ for this process by 
$\R_1(x,r, N,p)$ and $\R_2(x,r,p)$. 
That is
\[
\begin{split}
\R_1(x,r,N,p)
&=\displaystyle\sup_{0\le j\le j+q\le N-p} \left|\Cov_{\mu}(\mathbf 1_{B(x,r)},\mathbf 1_{\{S_{p+1}^{N-j}(x,r)=q\}})\right|, \\
\R_2(x,r,p) 
&= \mu(\{y\in B(x,r)\colon S_2^p(x,r)\ge1\}).
\end{split}
\]

\begin{proof}[Proof of Theorem~\ref{THM}] Let $t>0$.
Let $\sigma<d:=\dim_H\mu$ be such that $\sigma\alpha > \delta$ (with $\delta$ as in (\ref{couronne})).
For $\mu$-a.e. $x$, due to \eqref{eq:deltak}, by Proposition~\ref{propR1},
we have with $p_r=\lfloor r^{-\sigma}\rfloor$ and $N_r=\lfloor t/\mu(B(x,r))\rfloor$,
\[
\begin{split}
\R_1(x,r,N_r,p_r) &\le C'_0
\mu(B(x,r+r^\delta))p_r^{-\zeta+1}+ \mu(\mathcal C_{r,r^\delta}(x)) (1+N_r\mu(B(x,r+r^\delta)))\\
&= O(\mu(B(x,r)))r^{\sigma(\zeta-1)}+ o(\mu(B(x,r)))(1+t+o(1))\\
&=o(\mu(B(x,r)),
\end{split}
\]
due to (\ref{couronne}) and since $\zeta>1$ (see \eqref{poly}).
By Proposition~\ref{pro:R2} we have
\[
\R_2(x,r,p_r) = o(\mu(B(x,r))).
\]
The conclusion follows from Theorem~\ref{thm:poisson} by taking $M_r\to\infty$ sufficiently slowly.
\end{proof}
\section{Correlation estimates on towers}
We define the tower $\Delta$ and the map $F:\Delta\rightarrow\Delta$ as follow 
$$\Delta:=\{(x,\ell)\in\Lambda\times\mathbb Z\ :\ 0\le \ell<R(x)\},$$
$F(x,\ell)=(x,\ell+1)$ if $\ell<R(x)-1$ and $F(x,R(x)-1)=(f^R(x),0)$.
We define a family of partitions $(\mathcal Q_k)_{k\ge 0}$ by
$$\mathcal Q_0:=\{\Lambda_i\times\{\ell\},\ \ i\ge 1,\ \ell<R_i\} \ \
\mbox{and}\ \ \forall k\ge 1,\ \mathcal Q_k:=\bigvee_{i=0}^kF^{-i}\mathcal Q_0.$$
Now we define the projection $\pi:\Delta\rightarrow \M$ by
$\pi(x,\ell):=f^\ell(x)$.
Observe that $f\circ\pi=\pi\circ F$.
Due to Lemma 3.1 in \cite{aa}, there exists $C_2>0$ such that, for every non negative integer $k$
and every $Q\in \mathcal Q_{2k}$, we have
\begin{equation}\label{diametre}
\diam(\pi F^k(Q))\le C_2k^{-\alpha}. 
\end{equation}
Now we consider the quotient tower $\bar\Delta:=\Delta/\sim$
with $(x,\ell)\sim(y,\ell')$ if $\ell=\ell'$ and if $x$ and $y$
are on a same $\gamma^s\in\Gamma^s$. We define $\bar\pi:\Delta\rightarrow
\bar\Delta$ as the canonical projection and the map $\bar F:\bar\Delta\rightarrow\bar\Delta$ such that $\bar F\circ\bar\pi=\bar\pi\circ F$.
We define $\bar{\mathcal Q}_k$ as the projection of the partition
$\mathcal Q_k$ on $\bar\Delta$.
It will be also useful to consider the separation time $\bar s:\bar\Delta\times\bar\Delta\rightarrow\mathbb Z_+\cup\{\infty\}$  as follows:
$$ \bar s((x,\ell),(y,\ell'))=s(x,y) \ \mbox{if}\ \ell=\ell';\ 
  \bar s((x,\ell),(y,\ell'))=0 \mbox{ if }\ell\ne\ell'.$$
We fix a $\hat\gamma\in\Gamma^u$ and consider the measure $\bar m$ on $\bar\Delta$
such that the measure on $\bar\Delta_\ell$ corresponds to $\Leb_{\hat\gamma}$.
Recall that there exist probability measures $\mu$, $\nu$ and $\bar\nu$
on $\M$, $\Delta$ and $\bar\Delta$ respectively such that
$$f_*\mu=\mu,\ F_*\nu=\nu,\ \bar F_*\bar\nu=\bar\nu,\ \bar\nu=\bar\pi_*\nu,\
   \mu=\pi_*\nu $$
and such that
\begin{itemize}
\item the measure $\bar\nu$ admits a density function $\bar\rho$ with respect
to $\bar m$;
\item For every $c'>0$, there exists $C'$ such that, 
for every probability measure $\bar\lambda$ absolutely continuous
with respect to $\bar m$, with density probability $\varphi:\bar\Delta\rightarrow\mathbb R_+$ satisfying
$$\forall \bar Q\in \bar{\mathcal Q}_0,\ 
\forall x,y\in\bar Q,\ \ |\varphi(x)-\varphi(y)|\le c'|\varphi(x)|\beta^{\bar s(x,y)},$$
we have
\begin{equation}\label{VariationTotale}
\forall n\ge 1,\ \ d_{TV}(\bar F_*^n\bar\lambda,\bar\nu)
\le C' n^{-\zeta+1}\end{equation}
(see Theorem 3.5 in \cite{aa}).
\end{itemize}
This leads to the following decorrelation result at the core of our estimates.
\begin{lemma}\label{LEM0}
There exists $C'>0$ such that, for every non negative integers $k,n$
with $n\ge 2k$, for every $A$ union of atoms of $\mathcal Q_k$ and every
$B$ in $\sigma(\cup_{m\ge0}\mathcal Q_m)$, we have
$$|\Cov_{\nu}(\mathbf 1_A,\mathbf 1_{B}\circ F^n)|\le C' n^{-\zeta+1}\nu(A) .$$
\end{lemma}
\begin{proof}
Notice that $A=\bar\pi^{-1}\bar A$ and $B=\bar\pi^{-1}\bar B$ with
$\bar A:=\bar\pi A$ and $\bar B:=\bar\pi B$.
Therefore, we have
\begin{eqnarray*}
\left|\Cov_{\nu}(\mathbf 1_{A},\mathbf 1_{B}\circ F^{n})\right|
&=&\left|\Cov_{\bar\nu}(\mathbf 1_{\bar A},\mathbf 1_{\bar B}\circ 
\bar F^{n}))\right|\\
&=&\left|\mathbb E_{\bar m}\left[\bar\rho(\mathbf 1_{\bar A}-\bar\nu(\bar A))\mathbf 1_{\bar B}\circ \bar F^{n}\right]\right|\\
&=&\bar\nu(\bar A)\left|\mathbb E_{\bar m}\left[(\phi-\bar \rho)\mathbf 1_{\bar B}\circ \bar F^{n-k}\right]\right|\\
&\le&\nu(A)d_{TV}(\bar F_*^{n-k}(\bar F_*^{k}\bar\lambda),\bar\nu),
\end{eqnarray*}
writing $\bar\lambda$
for the measure whose density with respect to $\bar m$ is $\bar\rho\mathbf 1_{\bar A}/\bar\nu(\bar A)$
and $\phi$ for the density function of $\bar F_*^{k}\bar\lambda$.
Now, due to (20) in \cite{AlvesPinheiro} (see also Lemma 4.5 in \cite{aa}), 
we observe that 
$$\forall \bar Q\in\bar {\mathcal Q}_0,\ \forall y,z\in \bar Q,\ \ 
|\phi(y)-\phi(z)|\le \phi(y) e^{C_{\bar\rho}+C_{\bar F}}(C_{\bar\rho}+C_{\bar F})\beta^{\bar s(y,z)}$$
and, due to (\ref{VariationTotale}) and to $n-k\ge n/2$, we obtain
\begin{equation}\label{EQ3}
\left|\Cov_{\nu}(\mathbf 1_{A},\mathbf 1_{B}\circ F^{n})\right|\le C' \nu(A) n^{-\zeta+1},
\end{equation}
for some $C'>0$.
\end{proof}
\section{Decorrelations between successive returns}\label{sec:R1}

The aim of this section is to estimate $\R_1$.

\begin{proposition}\label{propR1}
There exists $C'_0>0$ such that for any integer $p$  we have
 $$\R_1(x,r,N,p)
\le C'_0 \mu(B(x,r+s))p^{-\zeta+1}+\mu(\mathcal C_{r,s}(x))(1+N\mu(B(x,r+s)) ),$$
where $\mathcal C_{r,s}(x)$ is the corona $B(x,r+s)\setminus B(x,r)$ and $s=C_2{\lfloor p/4\rfloor}^{-\alpha}$.
\end{proposition}
\begin{proof}
Let $K:=\lfloor p/4\rfloor $.
We have to estimate the following quantity:
$$\R_1(x,r,N,p)=\sup_{0\le j\le j+q\le N-p}
\left|\Cov_{\mu}(\mathbf 1_{B(x,r)},\mathbf 1_{\{S_{p+1}^{N-j}(x,r)=q\}})\right|.$$
Using successively $f_*\mu=\mu$, $\mu=\pi_*\nu$
and $f\circ\pi=\pi\circ F$, we obtain

$\displaystyle \left|\Cov_{\mu}(\mathbf 1_{B(x,r)},\mathbf 1_{\{S_{p+1}^{N-j}(x,r)=q\}})\right|=$
\begin{eqnarray*}
&=&\left|\Cov_{\mu}\left(\mathbf 1_{B(x,r)}\circ f^K,\mathbf 1_{\{S_{1}^{N-j-p}(x,r)=q\}}
    \circ f^{K+p}\right)\right|\\
&=&\left|\Cov_{\nu}\left(\mathbf 1_{B(x,r)}\circ \pi\circ F^K,
\mathbf 1_{\{S_{1}^{N-j-p}(x,r)=q\}}
    \circ \pi \circ F^{K+p}\right)\right|\\
&=&\left|\Cov_{\nu}\left(\mathbf 1_{F^{-K}\pi^{-1}(B(x,r))},
\mathbf 1_{F^{-K}\pi^{-1}(\{S_{1}^{N-j-p}(x,r)=q\})}
   \circ F^{p}\right)\right|.
\end{eqnarray*}
First we approximate the sets appearing in this last formula
with union of elements of the $\mathcal Q_k$'s.
More precisely, we approximate ${F^{-K}\pi^{-1}(B(x,r))}$ by ${D_{x,r}^{(K)}}$ with
$$D_{x,r}^{(K)}:=\bigcup_{Q\in \mathcal Q_{2K}:Q\cap F^{-K}\pi^{-1}B(x,r)\ne\emptyset } Q 
=\bigcup_{Q\in \mathcal Q_{2K}:\pi F^K Q\cap B(x,r)\ne\emptyset } Q $$
and ${F^{-K}\pi^{-1}(\{S_{1}^{N-j-p}(x,r)=q\})}$
by ${\{S_{1}^{N-j-p,(K)}(x,r)=q\}}$ where
$$ S_{1}^{k_0,(K)}(x,r)(y):=\#\{\ell=1,...,k_0\ :\ y\in D_{x,r}^{(K+\ell)}\}.$$
Let us set
$$A':=D_{x,r}^{(K)},\ \ A:=  F^{-K}\pi^{-1}B(x,r),\ \ A'':= 
   \bigcup_{Q\in \mathcal Q_{2K}:Q\cap (D_{x,r}^{(K)}\setminus  F^{-K}\pi^{-1}B(x,r))\ne\emptyset} Q, $$
$$B':=\{S_{1}^{N-j-p,(K)}(x,r)(y)=q\},\ \ B:=
F^{-K}\pi^{-1}(\{S_{1}^{N-j-p}(x,r)=q\})$$
and 
$$B'':=\bigcup_{\ell=1}^{N-j-p}
  \bigcup_{Q\in \mathcal Q_{2(K+\ell)}:Q\cap (D_{x,r}^{(K+\ell)}\setminus  F^{-K-\ell}\pi^{-1}B(x,r))\ne\emptyset} Q.$$
Observe that $A\subset A'\subset A\cup A''$ and
$B\subset B'\subset  B\cup B''$ so that we obtain

$\displaystyle|\Cov_\nu(\mathbf 1_A,\mathbf 1_B\circ F^{p})-\Cov_\nu(\mathbf 1_{A'},\mathbf 1_{B'}\circ F^{p})| \le $
\begin{eqnarray}
&\le& |\mathbb E[(\mathbf 1_{A'}-\mathbf 1_A)\mathbf 1_{B'}\circ F^p]+
\mathbb E[\mathbf 1_A(\mathbf 1_{B'}-\mathbf 1_B)\circ F^p]\nonumber\\
&\ &\ \ \ \ \ \ \  -(\nu(A')-\nu(A))\nu(B')-\nu(A)(\nu(B')-\nu(B))|\nonumber\\
&\le&
|\Cov_\nu(\mathbf 1_{A''},\mathbf 1_{B'}\circ F^{p})|+|\Cov_\nu(\mathbf 1_{A'},\mathbf 1_{B''}\circ F^{p})|+\nonumber\\
&\ &\ \ \ \ +2\nu(A'')\nu(B')+2\nu(A')\nu(B'').\label{EQ0}
\end{eqnarray}
Now, due to (\ref{diametre}) and (\ref{couronne}), we observe that
\begin{eqnarray}
\nu(A'')&\le& \nu(F^{-k}\pi^{-1}(B(x,r+s)\setminus B(x,r)))\nonumber\\
&\le&\mu(\mathcal C_{r,s}(x))\label{EQ1}
\end{eqnarray}
and that
\begin{eqnarray}
\nu(B'')
&\le& \sum_{\ell=1}^{N-j-p}\nu(D_{x,r}^{(K+\ell)}\setminus F^{-(K+\ell)}\pi^{-1}(B(x,r)))\nonumber  \\
&\le& \sum_{\ell=1}^{N-j-p}\nu(
F^{-(K+\ell)}\pi^{-1}(B(x,r+s)\setminus B(x,r))\nonumber  \\
&\le& N \mu(\mathcal C_{r,s}(x)).\label{EQ2}
\end{eqnarray}
Proposition \ref{propR1} follows from (\ref{EQ0}), (\ref{EQ1}), (\ref{EQ2}), $\nu(A')\le \mu(B(x,r+s))$, $\nu(B')\le 1$ together with Lemma \ref{LEM0}
applied three times (with $(A',B')$, $(A',B'')$ and $(A'',B')$ respectively).
\end{proof}

\section{Recurrence rate for tower systems}\label{sec:recrate}

Recall the definition of the recurrence rate 
\[
\rec(x)=\lim_{r\to0}\frac{\log\tau_r(x)}{-\log r},
\]
when the limit exists, where 
\begin{equation}\label{eq:tur}
\tau_r(x)=\inf\{n\ge1\colon d(f^nx,x)<r\}.
\end{equation}
A large lower bound for $\rec$ will allow a more efficient estimate of the second error term $\R_2$ in the next section.
For systems modeled by a Young tower as described in Section~\ref{sec:tower}, 
we show below that the recurrence rate is as large as it can be.

Since the measure $\mu$ is hyperbolic, its pointwise dimension exists a.e. \cite{bps} and we have
\begin{equation}\label{eq:dim}
\lim_{r\to0} \frac{\log\mu(B(x,r))}{\log r}=\dim_H\mu\quad\mu\text{-a.e..}
\end{equation}

In the case of super-polynomial decay of correlation, \cite{rs} applies and gives 
$\rec=\dim_H\mu$ $\mu$-a.e..
In our setting this decay might be only polynomial hence the general result above may not apply.
However, thanks to the Markov-tower structure we can refine the argument in \cite{rs}.

\begin{theorem}\label{thm:recrate}
For a system $(f,\mu)$ modeled by a Young tower we have $\rec=\dim_H\mu$ $\mu$-a.e.
\end{theorem}

First, we need the following decorrelation lemma (we use the notation $d:=\dim_H\mu$).
\begin{lemma}\label{lem:decball}
There exists $C>0$ such that for all $x$ and $r>0$ we have
\[
%\mu(B(x,r)\cap f^{-n}\{S_1^\ell(x,r)\ge1\}) \le C [(r+n^{-\alpha})^{d-\eps} n^{-\zeta+1}+\ell(r+n^{-\alpha})^{2d-2\eps}].
\mu(B(x,r)\cap f^{-n}\{S_1^\ell(x,r)\ge1\}) \le C [\mu(B(x,r+s)) n^{-\zeta+1}+\ell\mu(B(x,r+s))^2]
\]
where $s=\lfloor n/4\rfloor ^{-\alpha}$.
\end{lemma}
\begin{proof}
We define the set $A':=D_{x,r}^{(K)}$ as in the proof in Section~\ref{sec:R1}, but here we take $K=\lfloor n/4\rfloor$.
We have $F^{-K}\pi^{-1} B(x,r) \subset A'$.
In the same way, the set $F^{-K}\pi^{-1}\{S_1^\ell(x,r)\ge1\}$ is contained in $B':=\cup_{j=1}^\ell F^{-j}A'$.
Hence, due to Lemma \ref{LEM0},
\[
\begin{split}
\mu(B(x,r)\cap f^{-n}\{S_1^\ell(x,r)\ge1\}) 
&\le
\nu(A'\cap F^{-n}B') \\
&\le 
C'\nu(A') n^{-\zeta+1} + \nu(A')\nu(B').
\end{split}
\]
We remark that $\nu(A')\le \mu(B(x,r+K^{-\alpha}))$ since $\pi(F^K A')\subset B(x,r+K^{-\alpha})$
due to (\ref{eq:deltak}).
Moreover, by invariance of $\nu$, $\nu(B')\le \ell\nu(A')$, which finishes the proof.
\end{proof}

The proof of Theorem~\ref{thm:recrate} follows the lines of Lemma~16 in \cite{rs} but takes advantage of the (Markov) tower structure, 
which allows to use the much more efficient decorrelation Lemma~\ref{lem:decball} above
instead of an approximation by Lipschitz functions.

\begin{proof}[Proof of Theorem~\ref{thm:recrate}]

Let $\eps>0$ and take $\beta=\frac{1}{d-4\eps}$ with $\eps>0$ so small that $\beta\le\alpha$
and $1+2\eps\beta<\zeta$. 
Set
\[
G(\eps,\tilde r) = \{x\in \M\colon \forall r<\tilde r, r^{d+\eps}\le \mu(B(x,r)) \le r^{d-\eps}\}
\]
with $\tilde r$ so small that that $\mu(G(\eps,\tilde r))>1-\eps$ (see~\eqref{eq:dim}).

Choose $q=\lfloor 1/(\zeta-1-2\eps\beta) \rfloor$. Let $\ell_k= k^q $ and set $n_{k+1}=\ell_1+\cdots+\ell_k$.
We have $n_k\approx k^{q+1}$. Set $r_k=n_k^{-\beta}$. We have $n_k^{-\alpha}\le r_k$ since $\beta\le\alpha$.
By Lemma~\ref{lem:decball} we have for $x\in G(\eps,\tilde r)$ and $k$ sufficiently large so that $r_k+n_k^{-\alpha}<\tilde r$ 
\[
\mu(B(x,r_{k})\cap f^{-n_k}\{S_1^{\ell_{k}}(x,r_k)\ge1\}) 
\le 2^{2d}C  [r_k^{d-\eps}n_k^{-\zeta+1}+\ell_k r_k^{2d-2\eps}].
\]
Conditioning on $B(x,r_k)$ we get
\[
\mu(f^{-n_k}\{S_1^{\ell_{k}}(x,r_k)\ge1\}|B(x,r_{k})) 
\le 2^{2d}C  [r_k^{-2\eps}n_k^{-\zeta+1}+\ell_k r_k^{d-3\eps}].
\]
Let $T_k = \{y\in \M\colon \exists j=n_k+1,\ldots,n_{k+1}, d(f^j(y),y)<r_k/2\}$.
Take a maximal $r_k/4$-separated subset $G_k\subset G(\eps,\tilde r)$.
The balls $B(x,r_k/2)$, $x\in G_k$, cover $G(\eps,\tilde r)$.
Moreover the balls $B(x,r_k)$ cover it with a multiplicity bounded by some constant $C''$ 
depending only on the finite dimensional manifold $\M$.
Summing up over $x\in G_k$ we get
\[
\begin{split}
\mu(G(\eps,\tilde r) \cap T_k )
&\le
\sum_{x\in G_k} \mu(B(x,r_k/2) \cap T_k) \\
&\le \sum_{x\in G_k} \mu(B(x,r_k) \cap f^{-n_k}\{S_1^{\ell_{k}}(x,r_k)\ge1\}) \\
&\le
2^{2d}C C''[r_k^{-2\eps}n_k^{-\zeta+1}+\ell_k r_k^{d-3\eps}]\\
&\le
C''' (k^{(q+1)(1-\zeta+2\eps\beta)}+k^{q-(q+1)\beta(d-3\eps)}).
\end{split}
\]
This upper bound is summable in $k$ by construction.
An application of Borel-Cantelli lemma shows that for $\mu$-a.e. $x\in G(\eps,\tilde r)$ there exists $k(x)$
such that $x\not\in T_k$ for $k\ge k(x)$. For $x$ non periodic we have $\tau_r(x)\to\infty$. 
Hence there exists $r(x)$ such that $\tau_{r(x)}(x)>n_{k(x)}$. 
Let $r<\min(r(x),r_{k(x)})/2$. Take $k\ge k(x)$ such that $r_{k+1}<2r\le r_k$. 
Since $x$ does not belong to any of the $T_j$'s, $k(x)\le j\le k$, we get $\tau_r(x)\ge n_k$.
Hence
\[
\frac{\log \tau_r(x)}{-\log r} \ge \frac{\log n_k}{-\log r_{k+1}/2} \to\frac{1}{\beta} .
\]
Since $\eps$ is arbitrary this proves that for $\mu$-a.e. $x$
\[
\liminf_{r\to0}\frac{\log \tau_r(x)}{-\log r}\ge\dim_H\mu.
\]
On the other hand the limsup is always bounded by the dimension~\cite{bs}, 
whence the convergence of the recurrence rate and its equality with the dimension of $\mu$.
\end{proof}

\section{Estimate of short returns : $\R_2$}\label{sec:R2}

We now provide an optimal estimate of $\R_2$ in the following sense:
By Kac's Lemma the mean return time into $B(x,r)$ cannot be larger than $1/\mu(B(x,r))$.
Thus it is clear that any $\sigma>\dim_H\mu$ would contradict \eqref{eq:R2sigma}.

\begin{proposition}\label{pro:R2}
Let $\sigma<\dim_H\mu$.
For $\mu$-a.e. $x\in \M$ we have
\begin{equation}\label{eq:R2sigma}
\R_2(B(x,r),\lfloor r^{-\sigma}\rfloor ) =o(\mu(B(x,r))).
\end{equation}
\end{proposition}

\begin{proof}
Let $\eta>0$.
Let $r_1>0$ and define 
\[
K_{\sigma,r_1} = \{x\in \M\colon \forall r<r_1,\ \tau_{2r}(x)>r^{-\sigma}\}.
\]
Let $\iota>0$. By Theorem~\ref{thm:recrate} there exists $r_1>0$ such that $\mu(\M\setminus K_{\sigma,r_1})<\iota$.
Let $\tilde K_{\sigma,r_1}$ be the set of the Lebesgue density points of $K_{\sigma,r_1}$
with respect to the measure $\mu$. For $x\in\tilde K_{\sigma,r_1}$ we have
\[
\lim_{r\to0} \frac{\mu(B(x,r)\cap K_{\sigma,r_1})}{\mu(B(x,r))} = 1.
\]
Let $r_2<r_1$ and set 
\[
\tilde K_{\sigma,r_1,r_2} =\left\{x\in \M\colon \forall r<r_2, \frac{\mu(B(x,r)\cap K_{\sigma,r_1})}{\mu(B(x,r))} >1-\eta\right\}.
\]
Take $r_2>0$  so small that $\mu(K_{\sigma,r_1}\setminus K_{\sigma,r_1,r_2})<\iota$.
If $x\in \tilde K_{\sigma,r_1,r_2}$ and $r<r_2$ we get
\[
\begin{split}
\R_2(B(x,r),\lfloor r^{-\sigma}\rfloor)
&=
\mu(\{y\in B(x,r)\colon \tau_{B(x,r)}(y)\le r^{-\sigma}\}) \\
&\le 
\mu(\{y\in B(x,r)\colon \tau_{2r}(y)\le r^{-\sigma}\})\\
&\le 
\mu(B(x,r)\setminus K_{\sigma,r_1}) \\
&\le
\eta \mu(B(x,r)).
\end{split}
\]
This concerns all points in $\M$ except a subset of measure $2\iota$ arbitrary small, which proves the proposition. 
\end{proof}

\section{Application to solenoid with intermittency}

We can apply our main theorem to the example considered by \cite{AlvesPinheiro}:
Let $g\colon \TT^1\to\TT^1$ (with $\mathbb T^1=\mathbb R/\mathbb Z$) be a continuous map of degree $d\ge2$ such that 
\begin{itemize}
\item $g$ is $C^2$ on $\TT^1\setminus\{0\}$ and $g'>1$ on $\TT^1\setminus\{0\}$,
\item $g(0)=0$, $g'(0+)=1$ (right derivative at $0$) and there is $\gamma>0$ such that $-x f''(x)\sim|x|^{\gamma}$ in the right vicinity of $0$, 
\item $g'(0-)>1$ (left derivative at $0$).
\end{itemize} 
Let $J_0=(0,x_1)$, $J_1=(x_1,x_2)$,...,$J_{d-1}=(x_{d-1},0)$ be the successive intervals (from the left to the right)
on which $g$ defines a bijection onto $\mathbb T^1\setminus\{0\}$ (the $x_i$ are the non zero preimages of $0$ by $g$).
Let $\DD\subset \CC$ be the unit disk and consider the solid torus $\M = \TT^1\times\DD$.
We define the map 
\[f(x,z)=(g(x),\theta z+ \frac12e^{2i\pi x}),
\]
where $\theta\in(0,1)$ is such that $\theta\|g'\|_\infty<1-\theta$ (i.e. $0<\theta<1/(1+\Vert g'\Vert_\infty)$).

It was already proven by Alves and Pinheiro that this map fits into the general scheme described in Section~\ref{sec:tower} with the parameters $\zeta=1/\gamma>1$ and $\alpha=1+1/\gamma$. In particular the map $f$ admits an SRB measure $\mu$ if and only if $\gamma<1$.
 Let
\[
d((x,z),(x',z')) = \max( |x-x'|,|z-z'| ).
\]
\begin{theorem}\label{thm:solenoid}
For any $\gamma<\sqrt{2}/2$, the conclusion of Theorem \ref{THM} holds for $\mu$-almost every $x\in\mathcal M$,
i.e. for $\mu$-almost every $x\in\mathcal M$, the number of visits to $B(x,r)$ up to time $\lfloor t/\mu(B(x,r))\rfloor$
converges in distribution (with respect to $\mu$) to a Poisson random variable of mean $t$.
\end{theorem}
The only condition that is left to verify to prove Theorem \ref{thm:solenoid} is that \eqref{couronne} about coronas.
This condition is in principle highly dependent on the choice of the metric we put on $\M$.
We will prove it for $d$, which is the most natural metric on $\TT^1\times\CC$, but the proof could be adapted 
to the Euclidean metric as well.
\begin{proposition}\label{propSolenoide}
If $\gamma<\sqrt{2}/2$, for $\mu$-almost every $x\in\mathcal M$,
the assumption \eqref{couronne} on the coronas is satisfied.
\end{proposition}

The measure $\mu$ is supported on the attractor $\Lambda=\bigcap_{n\ge0} f^n\M$.
Let $\Z$ be the essential partition of $\M$ into $K_j=J_j\times\DD$ ($j=0,...,d-1$). 
This is a Markov partition for $f$. We denote by $\Z_{m}^n=\bigvee_{j=m}^{n-1}f^{-j}\Z$.
Its elements $Z\in\Z_m^n$ will be denoted indifferently by their code $[a_m,\ldots,a_{n-1}]$ meaning that $f^jZ\in K_{a_j}$ for any $j=m,\ldots,n-1$.

Let $\bar\mu$ be the marginal of $\mu$ on $\TT^1$. It is indeed the SRB measure of $g$, we denote its density $\bar h$. 
The partition (mod $\bar\mu$) $\bar\Z=\{J_0,\cdots,J_{d-1}\}$ is again a Markov partition for $g$.

Let $\bar\Z^n=\bar\Z\vee g^{-1}\bar\Z\vee\cdots\vee g^{-n+1}\bar\Z$. We denote indifferently by $[a_0\ldots a_{n-1}]\in \bar\Z^n$ the element $Z\in\bar\Z^n$ such that $g^jZ\subset J_{a_j}$ for any $j=0,\ldots,n$.

We collect below some elementary facts on the interval map $g$ needed to study its statistical properties.
\begin{proposition}[e.g. \cite{lsv,young99}]
(i) The density is uniformly controlled on cylinders:
\begin{equation}\label{H}
H:= \sup_{n} \sup_{ [0^n]\neq \bar Z\in\bar Z^{n}}\frac{\max_{\bar Z} \bar h}{\min_{\bar Z} \bar h}<\infty.
\end{equation}
(ii)
For any $n_0$ there exists a constant $D(n_0)>0$ such that:
for any integer $n\ge n_0$, for any $y,y' \in Z=[a_0\ldots a_{n-1}]\in \bar\Z^n$  such that $a_{n-n_0}\neq0$ we have
\begin{equation} \label{dis1}
\frac{(g^n)'(y')}{(g^n)'(y)} \le D(n_0) e^{D(n_0)|g^n(y')-g^n(y)|}.
\end{equation}
In particular for any subintervals $I,J\subset \bar Z$
\begin{equation}\label{dis2}
\frac{|g^nI|}{|g^nJ|} \le D(n_0) e^{D(n_0)} \frac{|I|}{|J|}.
\end{equation}

(iii) There exists a constant $D_1>0$ such that, for every nonnegative integers $m$ and $k$ and,
for every cylinder $ \bar Z=[0^ka_{k+1}\cdots a_{k+m}]\in\bar{\mathcal Z}^{m+k}$ with $a_{k+1}\ne 0$,
and any $z\in \bar Z$ we have
\begin{equation}\label{Ok1} 
\frac{k+1}{D_1}\le \bar h(z)\le D_1 (k+1),\ \ \frac{(k+1)^{-1-{1/\gamma}}}{D_1}\le\frac{| \bar Z|}{|[a_{k+1}\cdots a_{k+m}]|}\le D_1 (k+1)^{-1-{1/\gamma}}
\end{equation}
and so
\begin{equation}\label{Ok2}
\bar\mu( \bar Z)\approx (k+1)^{-{1/\gamma}}|[a_{k+1}\cdots a_{k+m}]|.
\end{equation}
\end{proposition}

\begin{lemma}\label{elln}
Let $p<\min(2,{1/\gamma})$. For $\mu$-a.e. $\xi\in \M$ we have 
\[
\sum_{j=0}^\ell \mathbf 1_{K_0} (f^{-n+j}(\xi)) \le \mu(K_0)\ell+o(n^{1/p}).
\]
\end{lemma}
\begin{proof}
By Theorem~1.6 in \cite{dgm} we have 
\[
\sum_{n=0}^\infty \frac1n \bar\mu\left(\left\{ x\colon \max_{1\le k\le n}\left|\sum_{j=0}^k \mathbf 1_{J_0}(g^jx)-k\bar\mu(J_0)\right|>n^{1/p}
\right\}\right)<\infty
\]
and using the invariance of $\mu$ we get that 
\[
\sum_{n=0}^\infty \frac1n \mu\left(\left\{ \xi\colon \max_{1\le k\le n}\left|\sum_{j=0}^k \mathbf 1_{K_0}(f^{-n+j}\xi)-k\mu(K_0)\right|>n^{1/p}\right\}\right)<\infty.
\]
We now conclude accordingly that for $\mu$-a.e. $\xi$
\[
\limsup_{n}  \frac{|\sum_{j=0}^k \mathbf 1_{K_0}(f^{-n+j}\xi)-k\mu(K_0)|}{n^{1/p}}<\infty.
\]
The conclusion follows since $p<{1/\gamma}$ was arbitrary.
\end{proof}

{} From now on we fix $p<{1/\gamma}$ and $\xi=(x,z)\in \Lambda$ satisfying the conclusion of Lemma~\ref{elln}.
The code $\Z_{-n}^0(\xi)=[a_{-n},\ldots,a_{-1}]$ is such that $\ell_n=o(n^{1/p})$, where $\ell_n$ is the minimal integer $\ell\ge0$ such that $a_{-n+\ell}\neq0$. 

We also take $n_0\ge1$ the minimal integer such that $a_{-n_0}\neq0$.

An immediate computation gives
\begin{equation}\label{eq:fnxz}
f^n(x,z) = (g^n(x),\theta^n z + \frac12 \sum_{j=0}^{n-1} \theta^{n-1-j} e^{2i\pi g^j(x)}).
\end{equation}
We consider the natural projections $\pi_X: \mathcal M\rightarrow\mathbb T^1$ and 
$\pi_{\mathbb D}: \mathcal M\rightarrow\mathbb D$.
\begin{lemma}\label{cylindreout}
For any $\xi=(x,z)\in\M\setminus\partial\Z$ and any $r<|x-0|$ we have 
\[
f^{n+k_0}(\M) \cap B(\xi,r) \subset \Z_{-n}^0(\xi).
\]
with $n=\lfloor \log(r)/\log(\theta)\rfloor-k_0$, for some constant $k_0$ independent of $\xi,r$.
\end{lemma}
Observe that, with the notations of this lemma, we have $\theta^{n+k_0+1}<r\le\theta^{n+k_0}$.
\begin{proof}[Proof of Lemma \ref{cylindreout}]
\begin{claim}
Let $\xi'\in f^n(\M)$ such that $d(\xi,\xi')<r$.
There exists $\xi''\in\Z_{-n}^0(\xi')$ such that $\pi_X\xi=\pi_X\xi''$ and $|\pi_\DD\xi-\pi_\DD\xi''|<4r$.
\end{claim}
\begin{proof}
Taking $y$ the preimage by $g^n$ of $\pi_X\xi$ lying in the same branch of $g^n$ than $\pi_Xf^{-n}(\xi')$ gives
a point $\xi''=f^n(y,\pi_\DD(f^{-n}(\xi)))\in\Z_{-n}^0(\xi')$ such that $\pi_X\xi=\pi_X\xi''$ and 
$|\pi_\DD\xi-\pi_\DD\xi''|<\frac{\pi}{1-\theta}r<4r$. 
\end{proof}

\begin{claim}
There exists a constant $k_0$ such that $\pi_Xf^{-n+k_0}\xi=\pi_Xf^{-n+k_0}\xi''$, in particular $\xi''\in\Z_{-n+k_0}^0(\xi)$, where $n$ is the smallest integer such that $r<\theta^n$.
\end{claim}
\begin{proof}
Set $\xi''=(x,z'')$ with $|z-z''|<4r$.
Remark that if $g(u)=g(v)$ and $u\neq v$ we have $|e^{2i\pi u}-e^{2i\pi v}|\ge
2\sin(\pi|u-v|)\ge 2\sin(\frac{\pi}{\|g'\|_\infty})$.
Set $(x_{-n},z_{-n})=f^{-n}\xi$ and $(x_{-n}'',z_{-n}'')=f^{-n}\xi''$. 
Let $k$ denote the largest integer $\le n-1$ such that $g^k(x_{-n})\neq g^k(x''_{-n})$ (if any).
We have 
\[
4r>|z-z''| = \left|\theta^n(z_{-n}-z_{-n}'')+\frac12 \sum_{j=0}^{k} \theta^{n-1-j} (e^{2i\pi g^j(x_{-n})}-e^{2i\pi g^j(x_{-n}'')})\right|
\]
from which we get
\[
\theta^{n-1-k}\sin(\frac{\pi}{\|g'\|_\infty}) < 8r + 4\theta^n + \frac{\theta^{n-k}}{1-\theta}
\]
and thus
\[
\frac{1}{\|g'\|_\infty}\le \sin(\frac{\pi}{\|g'\|_\infty}) < 12\theta^{k+1}+ \frac{\theta}{1-\theta}.
\]
Taking $k_0$ the smallest integer $k$ such that the above formula does not hold shows that 
the section $\{x\}\times D_{4r}(z)$ is contained in a unique cylinder $\Z_{-n+k_0}^0(\xi)$.
\end{proof}
A change of indices finishes the proof of the lemma.
\end{proof}
Let us write $D(z,r):=\{z'\in\mathbb C\ :\ |z-z'|<r\}$ for the open disc of center $z$ and radius $r$ in 
$\mathbb C$. If $ {\mathcal Z}_{-j}^0(\xi)=[a_{-j}\cdots a_0]$, we set $(g^{-j})_{\xi}$
for the inverse branch of $g^j$ restricted to $\bar {\mathcal Z}_{0}^j(\pi_Xf^{-j}\xi)$.

\begin{lemma}
The cylinder $\Z_{-n}^0(\xi)$ is the tube 
\[
\Z_{-n}^0(\xi) = \{(x,z)\in\TT\times\D\colon x\in]0,1[, z\in D(\gamma_{\xi,n}(x),\theta^n)\}.
\]
around  the curve  $\gamma_{\xi,n}\colon ]0,1[\to\CC$ defined by 
\[
\gamma_{\xi,n}(x) = \frac12 \sum_{j=1}^n \theta^{j-1} e^{2i\pi (g^{-j})_\xi(x)}
\]
with slope $\gamma_{\xi,n}'(x)=u_n(\gamma_{\xi,n}(x))$ where 
\[
u_n(\xi'):= i\pi \sum_{j=1}^n\theta^{j-1} \frac{1}{(g^j)'(\pi_X f^{-j}\xi')} e^{2i\pi \pi_Xf^{-j}\xi'}.
\]
\end{lemma}
\begin{proof}
A change of indices in Expression~\ref{eq:fnxz} shows the first assertion, the second follows by derivation.
\end{proof}
For every $\xi'=(x',z')\in\Lambda$, we also define $u_\infty(\xi')$ as follows
$$u_\infty(\xi')= i\pi \sum_{j\ge 1}\theta^{j-1} \frac{1}{(g^j)'(\pi_X f^{-j}\xi')} e^{2i\pi \pi_Xf^{-j}\xi'}.$$
Observe that $(1,u_\infty(\xi'))$ corresponds to the direction of the solenoid at $\xi'$.
\begin{lemma} There exist $c>0$, $c_0>0$, $c_1>0$ and $C(n_0)>0$ such that for any $n\ge n_0$ and $r<|x-0|$, we have

(i) The direction $u_n(\xi)$ is uniformly bounded away from zero and infinity: $c_0<|u_n(\xi)|<c_1$.

(ii) For any $\xi'\in \Z_{-n}^0(\xi)$ such that $|\pi_X\xi-\pi_X\xi'|<r$, and any $k\ge0$, 
the directions satisfy
\[
|u_n(\xi)-u_{n+k}(\xi')|\le C(n_0)r+\frac{\pi\theta^n}{1-\theta}.
\]
\end{lemma}
\begin{proof}
(i) The upper bound $c_1=\pi/(1-\theta)$ is obvious since $|g'|>1$. 
Isolating the $j=1$ term gives the lower bound
\[
|u_n(\xi)| \ge \left|\frac{\pi}{g'(\pi_X(f^{-1}\xi))}\right|-\frac{\pi\theta}{1-\theta}\ge \pi
\left(\frac{1}{\|g'\|_\infty}-\frac{\theta}{1-\theta}\right)=:c_0.
\]

(ii) The sum in $u_{n+k}(\xi')$ differs from that of $u_n(\xi')$ only by the last terms $j=n+1,\ldots,n+k$, 
and this is bounded by the rest of the convergent series, so it suffices to prove the claim for $k=0$.
The distortion bound \eqref{dis1} yields, since $|\pi_X f^{-j}\xi'-\pi_X f^{-j}\xi|<r$ for all $j=1,\ldots,n$,
\[
\left|\frac{(g^j)'(\pi_X f^{-j}\xi) e^{2i\pi \pi_Xf^{-j}\xi'}}{(g^j)'(\pi_X f^{-j}\xi') e^{2i\pi \pi_Xf^{-j}\xi}}-1\right| \le D(n_0)e^{D(n_0)r+2\pi r}.
\]
\end{proof}

We now prove Proposition \ref{propSolenoide}.

\begin{proof}
Let $\xi=(x,z) \in \M$ with $x\neq0$. 
Since
\[
B(\xi,r) = (x-r,x+r)\times D(z,r),
\]
setting $\mathcal{C}_{r,s}(x)$ and $\mathcal{C}_{r,s}(z)$ the coronas in $\TT^1$ and $\CC$ respectively, we get
\[
\C_{r,s}(\xi) =(\C_{r,s}(x) \times D(z,r+s)) \cup ((x-r,x+r)\times \C_{r,s}(z)).
\]
The first term is roughly bounded by the help of the marginal measure $\bar\mu$, whose density is bounded in a neighborhood of $x\neq0$, hence 
\[
\mu(\C_{r,s}(x) \times D(z,r+s)) \le \bar \mu(\C_{r,s}(x))\le c s,
\]
provided $r$ is sufficiently small and $s<r$.
For typical $\xi$ the decay of $\mu(B(\xi,r))$ is ruled by the dimension $\dim_H\mu$ (see\eqref{eq:dim}), hence choosing 
$\delta=\frac{1+\alpha}{2}\dim_H\mu$ and $s=r^\delta$ the last bound  satisfies $c s =o(\mu(B(\xi,r)))$.

The second term deserves more attention. It suffices to show that 
\[
\mu(W)= o(\mu(B(\xi,r))),\ \ \mbox{where}\ \ W := (x-r,x+r)\times \C_{r,s}(z) .
\]
Without loss of generality we will suppose, for this part of the proof only, that $s=r^\delta$ for some fixed $\delta$,
$1<\delta <\min(2,\alpha\dim_H\mu)$. Let $k$ be such that $\theta^{n+k}=s$. 
Let $\beta=\frac srn^\nu$ for some exponent $\nu$ that will be fixed later. 
Note that $r^2 \ll s \ll r\ll \beta$. 
Recall that $\ell_n$ is the minimal integer $\ell\ge 0$ such that $a_{-n+\ell\ne 0}$.
We let $Z=\Z_{-n}^0(\xi)$ and set $U:=g^{\ell_n}\pi_X(f^{-n}Z)\in\bar{\mathcal Z}^{n-\ell_n}$.
Note that the code of $U$ starts with a non zero symbol.
To simplify the exposition we introduce the notation
\[
 I \ltimes B := \{(t,z)\colon t\in I, (t,z)\in B\} =\pi_X^{-1}(I)\cap B,
\]
for $I\subset\TT^1$ and $B\subset \M$.
We decompose $W$ with the pieces of element $Z'$ of the partition $\Z_{-n-k}^0$ which intersect it.  
\begin{lemma}
Each $Z' \cap W$ (with $Z'\in\mathcal Z_{-n-k}^0$) is contained in a tube $\mathcal T_{Z'}$
\[
Z'\cap W \subset I \ltimes Z' \subset \mathcal T_{Z'}:=\{(t,z)\in \M\colon t\in I,|z-z_0-(t-t_0)u_{Z'}|<D_1s\}
\]
where $D_1=2+C(n_0)+\pi/(1-\theta)$, $I\subset (x-r,x+r)$, $(t_0,z_0)\in Z'\cap W$ and $|u_{Z'}-u_n(\xi)|<c r$.
\end{lemma}

There are two type of intersections, transversal and non transversal (See Figure~\ref{tasseetsapghetti}).

\begin{figure}[h]
\begin{center}
\includegraphics[scale=0.5, trim= 3cm 4cm 4cm 3cm, clip]{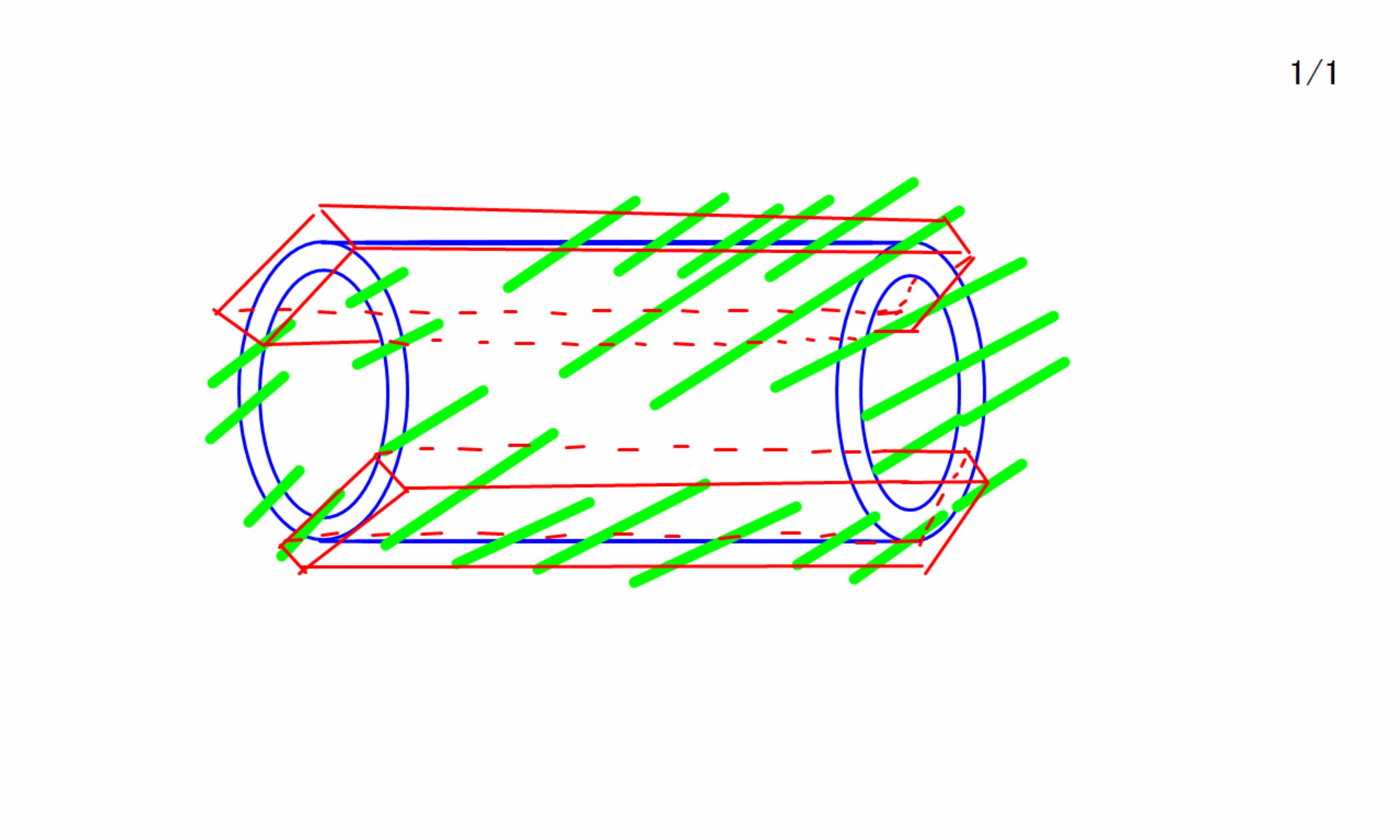}
\end{center}
\caption{The green rods represent the pieces of cylinders $Z'$ intersecting $W$. 
The non transversal intersections happen in the red boxes.}
\label{tasseetsapghetti}
\end{figure}

Let $v_n=u_n(\xi)/|u_n(\xi)|$.
We say that $Z'$ intersects $W$ non transversally if there exists $\xi'=(x',z')\in Z'\cap\Lambda$ such that
the angle between $u_\infty(\xi')$ and the circle of center $z$ and radius $|z-z'|$ is smaller than $\beta$.
But since $\angle (v_n,u_\infty(\xi'))\precsim r\ll\beta$, such a transversal intersection can only occur in
$$(x-r,x+r)\times\{z'\in\mathbb C\ :\ z'=z+(\pm r+q)v_ne^{ip},\ |q|<s,\ |p|\precsim\beta\}.$$
Now, for any $\xi''\in ((x-r,x+r)\ltimes Z')\cap\Lambda$, we have $\angle(u_\infty(\xi''),u_\infty(\xi'))|
  \precsim r\ll\beta$.
So for such a $Z'$, the intersection $Z'\cap\Lambda\cap W$ is contained in one of the boxes $NT$
\[
NT = (x-r,x+r)\times \{z+ p v_n + (\pm r+q) iv_n\colon |p|<2r\beta,|q|<s\}
\]
(since $r|\sin\beta|\le r\beta$ and $r(1-\cos\beta)\le r\beta^2\ll s$).
These two boxes are parallel to the plane defined by the $x$-axis
and the vector $u_n(\xi)$, of length $r$, height of order $r\beta$ and thickness $s$. There are at most $C r/s$ different $Z'$ intersecting $W$ non transversally, since they are separated by a distance of order $s$ and they point in the same direction bounded away from the horizontal axis (the $x$-axis). Moreover, each $Z'\cap W$ is contained in a set $I_{Z'}\ltimes Z'$ for some interval $I_{Z'}$ of length at most $cr\beta$.
\begin{lemma}
For any interval $I\subset \mathbb T^1\setminus\{0\}$ and any cylinder $Z'\in\Z_{-n-k}^0$ such that $Z'\subset Z$ we have 
\[
\mu(I\ltimes Z') \approx \mu(Z') |I|.
\]
\end{lemma}
\begin{proof}
The set $f^{-n-k}(I\ltimes Z')$ is of the form $J\times\DD$ where the interval $J$ 
is a subset of the cylinder $\bar Z'=\pi_X(f^{-n-k}Z)\in\bar\Z^{n+k}$ and $g^{n+k}J=I$. 
By invariance we have $\mu(I\ltimes Z') = \bar\mu(J)$.
Observe that $\bar Z'$ and $\bar Z$ have the same code. 
By distortion~\eqref{dis2} we have 
\[
\frac{|J|}{|\bar Z'|} \approx \frac{|I|}{|\TT^1|}.
\]
This ratio of length transfers to a ratio of measures by~\eqref{H}, proving the lemma
since by invariance again $\bar\mu(\bar Z')=\mu(Z')$.
\end{proof}

Note that except the very particular cylinder $Z'=[0^kZ]\in\mathcal Z_{-n-k}^0$ which has a measure 
\[
\mu([0^kZ]) \approx (k+\ell_n)^{-{1/\gamma}} |U| \approx k^{-{1/\gamma}}|U|, 
\]
(due \eqref{Ok2}), 
all the other $Z'\in\mathcal Z_{-n-k}^0$ contained in $Z$ 
do have at least a non zero symbol in their code at some position $-n-m_0\in\{-n-k,\ldots, -n-1\}$. Applying
 \eqref{Ok2} to them, we obtain
\[
\begin{split}
\mu(Z') &\approx  (k-m_0+1)^{-{1/\gamma}}|[a_{-n-m_0}\cdots a_0]|\\
&\precsim  (k-m_0+1)^{-{1/\gamma}} (m_0+\ell_n)^{-1-{1/\gamma}} |U| \\
&\precsim 
( k^{-{1/\gamma}}(1+\ell_n)^{-1-{1/\gamma}} +  k^{-1-{1/\gamma}})|U|
\end{split}
\] 
(considering separately the cases $m_0\le k/2$ and $m_0>k/2$).
Therefore the total measure coming from non transversal intersections is bounded by
\begin{equation}\label{EEE1}
\mu( \bigcup_{Z'\cap NT\neq\emptyset} Z')
\precsim [ k^{-{1/\gamma}} +   \frac rs ( k^{-{1/\gamma}}(1+\ell_n)^{-1-{1/\gamma}} +  k^{-1-{1/\gamma}})]|U|cr\beta.
\end{equation}

For transversal intersections, the angle is at least $\beta$. Thus $Z'\cap W \subset I_{Z'} \ltimes Z'$ for some
interval $I_{Z'}$ of length less than $s/\beta$.
Hence the measure of the intersection $\mu(Z'\cap W) \precsim \frac{s}{\beta} \mu(Z')$. Therefore
\begin{equation}\label{EEE2}
\mu(\bigcup_{Z'\not\in NT} Z') \precsim \frac{s}{\beta} \mu(Z) \precsim \frac{s}{\beta} c(1+\ell_n)^{-{1/\gamma}}|U|.
\end{equation}
Finally, the ball $B(\xi,r)$ contains the set $I\ltimes Z_*$ for some cylinder $Z_*\in\Z_{-n-2}^0$, $Z_*\subset Z$,
and an interval $I$ of length say $r/4$. Thus
\begin{equation}\label{EEE3}
\mu(B(\xi,r)) \ge \mu(I\ltimes Z_*) \succsim  \mu(Z_*) r \succsim  (1+\ell_n)^{-1-{1/\gamma}}|U| r.
\end{equation}
Putting \eqref{EEE1}, \eqref{EEE2} and \eqref{EEE3} together, we obtain
\[
\frac{\mu(W)}{\mu(B(\xi,r))} \precsim (I) + (II) +(III) +(IV)
\]
where 
\[
\begin{split}
(I) &= k^{-{1/\gamma}}\beta (1+\ell_n)^{1+{1/\gamma}}=\frac sr O(n^{-{1/\gamma}+\nu+(1+{1/\gamma})/p})\\
(II) &= k^{-{1/\gamma}}\frac{r\beta}{s} = O(n^{-{1/\gamma}+\nu})\\
(III) &= k^{-1-{1/\gamma}} \frac{r\beta}{s} (1+\ell_n)^{{1/\gamma}+1} = O(n^{-1-{1/\gamma}+\nu+({1/\gamma}+1)/p})\\
(IV) &=  \frac{s}{r\beta}(1+\ell_n) = O(n^{-\nu+1/p}).
\end{split}
\]
\end{proof}
The first term $(I)$ goes to zero exponentially fast in $n$.
The second term $(II)$  goes to zero provided $\nu<{1/\gamma}$, 
the third one provided $\nu<{1/\gamma}+1-({1/\gamma}+1)/p$ 
and the last one provided $\nu>1/p$. If $\gamma>1/2$, 
all these conditions are satisfied whenever the interval $\left(\gamma,1/\gamma-\gamma\right)$ is nonempty
(i.e. whenever $\gamma<1/\sqrt{2}$), 
taking $\nu$ in the interval and $p$ sufficiently close to ${1/\gamma}$.
If $\gamma\le1/2$, $(I)$, $(II)$, $(III)$ and $(IV)$ go to zero provided $p$ is close enough to $2$.
\section{Billiard in a stadium}
Let $\ell>0$. We consider the convex domain $Q$ of the plane corresponding to the union of the rectangle
$[-\ell/2,\ell/2]\times[-1,1]$ with the two discs of radius 1 centered at $(-\ell/2,0)$ and $(\ell/2,0)$
respectively. This domain $Q$ is called stadium.

\begin{figure}[h]
\begin{center}
\includegraphics[scale=0.5]{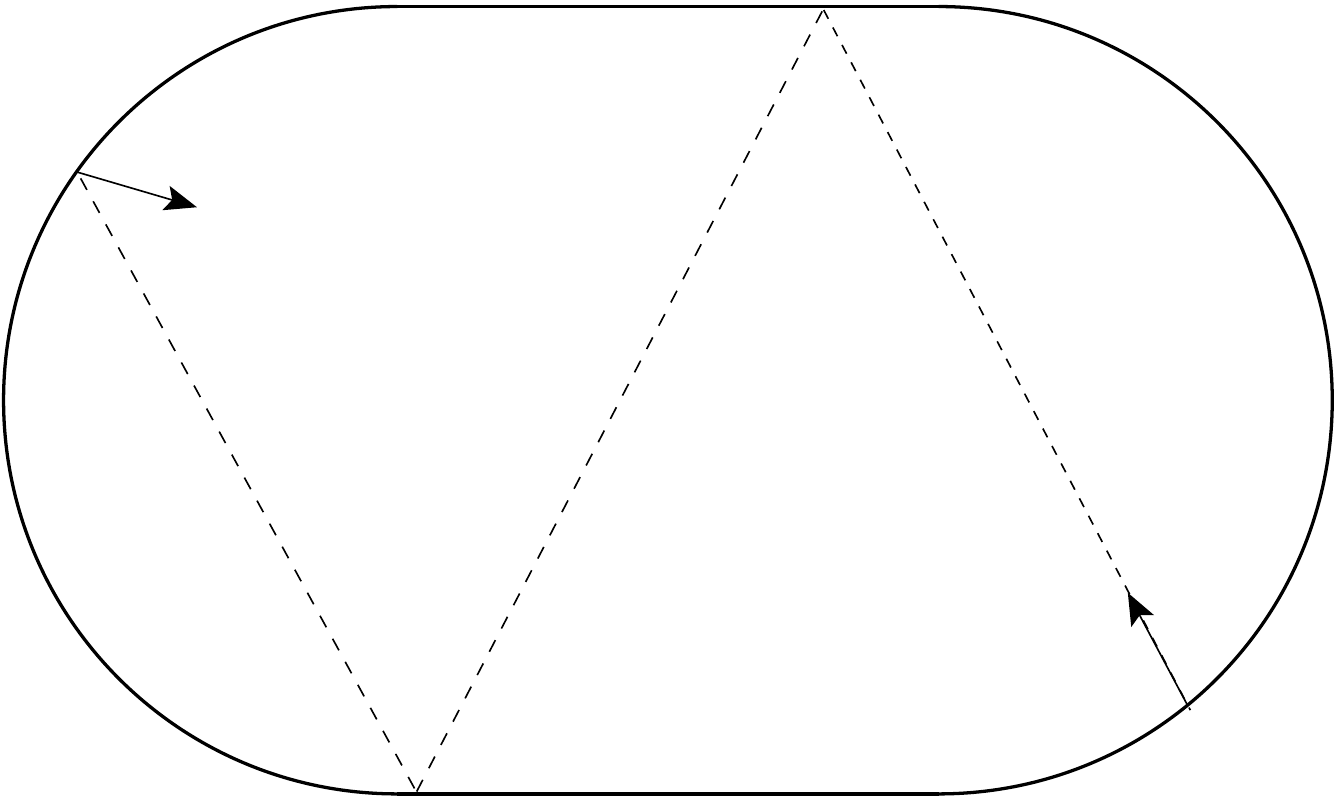}
\end{center}
\caption{The stadium billard}
\label{thestadium}
\end{figure}
The billiard system $(\mathcal M,f,\mu)$ will describe (at reflection times) the evolution of a point particle moving in $Q$
with elastic reflection off $\partial Q$ and going straight on between two reflection times.
The phase space $\mathcal M$ corresponds to the set of unit vectors based on $\partial Q$ and pointing inwards. We define it as follows:
$$\mathcal M:=\frac{\mathbb R}{2(\pi+\ell)\mathbb Z}\times\left]-\frac\pi 2,\frac \pi 2\right[. $$
A particle has configuration $x=(r,\varphi)\in\mathcal M$ if its position $p\in\partial Q$ has curvilinear abscissa $r$
($\partial Q$ being oriented counterclockwise and $r=0$ corresponding to the point $(\ell/2,-1)$) and if its reflected vector makes
the angle $\varphi$ with $\vec n(p)$, where $\vec n(p)$ is 
the unit normal vector to $\partial Q$ at $p$ oriented inwards.

The transformation $f:\mathcal M\rightarrow\mathcal M$ maps a configuration
at a reflection time to the configuration at the next reflection time.

\begin{figure}[h]
\begin{center}
\includegraphics[scale=0.5]{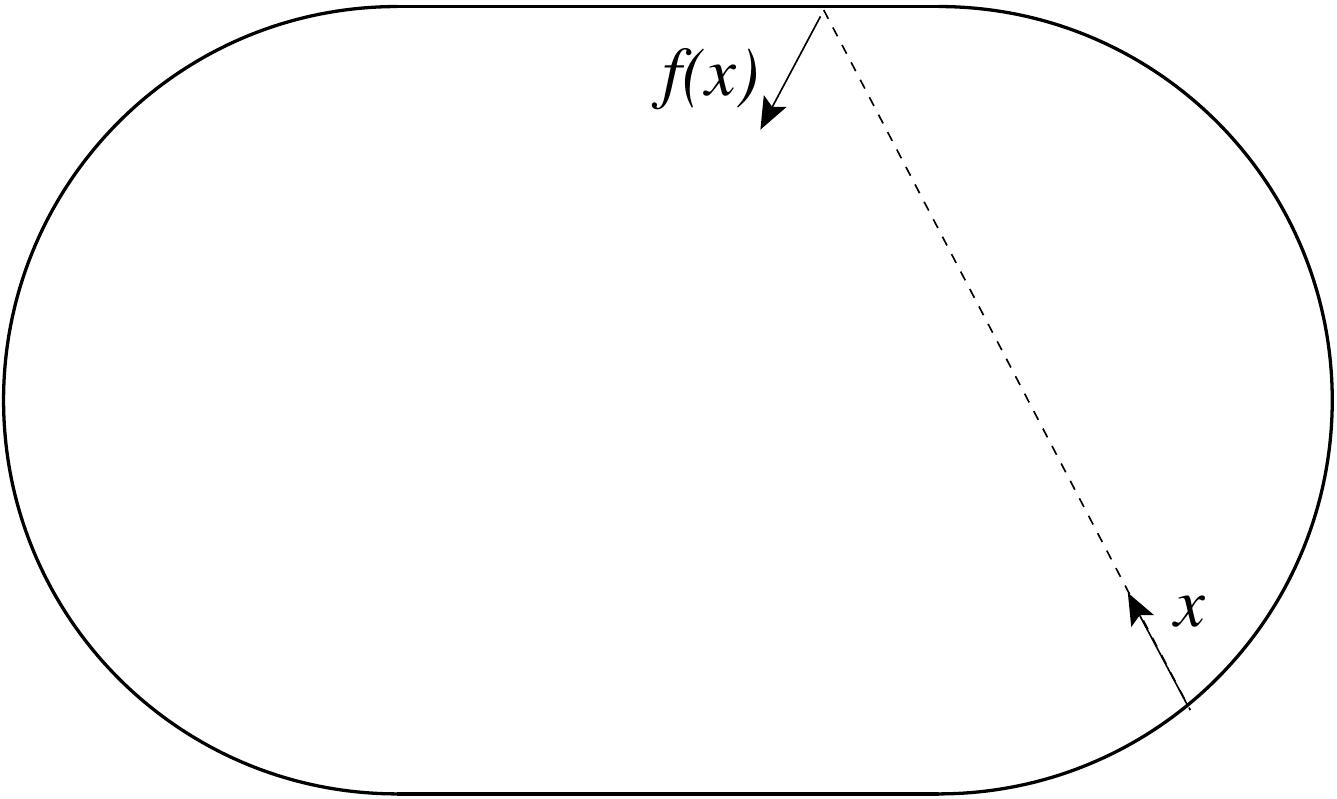}
\end{center}
\caption{The billard map $f$.}
\label{themap}
\end{figure}
This transformation preserves the probability measure $\mu$ on $\mathcal M$ given by
$$d\mu=\frac{\cos\varphi}{4(\pi+\ell)}\, drd\varphi.$$
Here again we consider the metric $d((r,\varphi),(r',\varphi')):=\max
(|r-r'|,|\varphi-\varphi'|)$ (but our result also holds for the other usual metrics).
\begin{theorem}\label{thmbilliard}
For the billiard system $(\mathcal M,f,\mu)$ in the stadium described above,
the conclusion of Theorem \ref{THM} holds true for $\mu$-almost every $x\in\mathcal M$, i.e. for $\mu$-almost every $x\in\mathcal M$, the number of visits to $B(x,r)$ up to time $\lfloor t/\mu(B(x,r))\rfloor$
converges in distribution (with respect to $\mu$) to a Poisson random variable of mean $t$.
\end{theorem}
\begin{proof}
The fact that $(\mathcal M,f,\mu)$ can be modeled by a Young tower satisfying 
(P1)--(P5) can be found in \cite{Markarian04,ChernovZhang05,BalintGouezel06}.
Namely, the fact that (P2)--(P3) hold with $\alpha=1$ (and so \eqref{eq:deltak}) comes from Propositions 2.1 and 2.3 of \cite{BalintGouezel06} and the fact that 
\eqref{poly} holds with $\zeta=2$ is proved in Section 9 of 
\cite{ChernovZhang05}.
Finally, because of the continuity and positivity of the density function of $\mu$ with respect to the Lebesgue measure, \eqref{couronne} holds for every $x\in\mathcal M$ (See the Appendix for details).
\end{proof}

\appendix

\section{Measure of coronas}

In this section we discuss various conditions under which Assumption~\eqref{couronne} holds and more precisely on
the following condition: $\exists C_0,r_0>0$ such that 
\begin{equation}\label{couronne2}
\forall 0<s<r,\quad \mu(B(x,r+s)\setminus B(x,r))\le C_0 s^ar^{-b}\mu(B(x,2r)).
\end{equation}
Indeed, we observe that, if $\alpha\dim_H\mu>b/a$, (\ref{couronne2}) 
implies (\ref{couronne}) for any $\delta\in(b/a,\alpha\dim_H\mu)$.

\begin{proposition}
Suppose that $\mu$ is absolutely continuous with respect to the Lebesgue measure.
\begin{itemize}
\item[(i)] For any continuity point $x$ of the density, we have (\ref{couronne2}) with $a=b=1$
and so (\ref{couronne}) since $\alpha \dim_H\mu>1$.
\item[(ii)] If the density is in $L^p_{\mathrm{loc}}$ for some $p>1$ then Assumption~\eqref{couronne2} holds a.e. with
$a=1/q$ and $b=(1+d(q-1))/q$, where $1/p+1/q=1$.
\end{itemize}
\end{proposition}
\begin{proof}
(ii) Suppose that the density $h$ is in $L^p$ on some open set $V$ and take $x\in V$ with $B(x,2r)\subset V$. 
By the H\"older inequality for any $s<r$
\[
\mu(B(x,r+s)\setminus B(x,r)) = \int_V h\mathbf{1}_{B(x,r+s)\setminus B(x,r)}\le \|h\|_{L^p(V)} C ((r+s)^d-r^d )^{1/q}.
\]
The upper bound is of order $r^{(d-1)/q} s^{1/q}$ while the measure of the ball is lower bounded by $Const\times r^d$
for a.e. $x$.
\end{proof}

In particular for Billiard systems one can take $a=b=1$. \\

\begin{proposition}
If $\mu$ is non-atomic and satisfies \eqref{couronne2} at some $x\in\supp(\mu)$, then $b\ge a$.
\end{proposition}
\begin{proof}
Assume that $\mu$ satisfies \eqref{couronne2} with $b<a$. Let $\eta\in(0,a-b)$. 
Then there exists $n_0$ such that, for every integer $n\ge n_0$, we have
$$\mu(B(x,2^{-n})\setminus B(x,2^{-n-1}))\le 2^{-\eta n}\mu(B(x,2^{-n})),$$
i.e. 
$(1-2^{-\eta n})\mu(B(x,2^{-n}))\le \mu(B(x,2^{-n-1}))$. Therefore, for any $n> n_0$, we have
$$\mu(B(x,2^{-n_0}))\prod_{k=n_0}^{n-1} (1-2^{-\eta n})\le \mu(B(x,2^{-n}))$$
and so $\mu(\{x\})=\lim_{n\rightarrow +\infty}\mu(B(x,2^{-n}))>0$ if $\mu(B(x,2^{-n_0}))>0$.
\end{proof}

Now let us prove Proposition~\ref{pro:subseq}
\begin{proposition}\label{pro:subseq2}
Let $\theta<1$.
For any $a\in(0,1)$, for every $x\in\mathcal M$, there exists a constant $C_0$ such that 
there exists a sequence $r_n(x)\in(\theta^{n+1},\theta^{n})$ such that, for all $n$, \eqref{couronne2} holds with $b=a$ for the sequence $r=r_n(x)$.
\end{proposition}
\begin{proof}
Up to a change of $\theta$ in one of its integer roots, we assume without loss of generality that $1/2<\theta<1$.
Fix $x\in\M$ and $n\in\NN$. Let $\lambda<1/2$ and set $\rho=(1-2\lambda)/3$.
We define a measure on $[0,1]$ by setting $m([0,t])=\mu(B(x,\theta^{n+1}+t(\theta^n-\theta^{n+1})))-\mu(B(x,\theta^{n+1}))$ for $t\in[0,1]$.
We construct a sequence of nested intervals $I_k$ by dichotomy as follows. Let $I_0=(0,1)$.
Given $I_k$ (for some $k\ge 0$), we consider two intervals of length $\lambda |I_k|$ included
in $I_k$, at distance $\rho|I_k|$
from the left and right endpoints of $I_k$ respectively and we define $I_{k+1}$ as the one with smallest measure. 
We have $|I_k|=\lambda^k$ and $m(I_k)\le 2^{-k} m(I_0)$.
Let $\{t_*\}$ be the intersection of the $I_k$'s. By construction, since $t_*\in I_{k}$ the $\rho\lambda^k$-neighborhood of $t_*$
is contained in $I_{k-1}$. Let us prove that the radius $r=\theta^{n+1}+t_*(\theta^n-\theta^{n+1})$ satisfies \eqref{couronne2} with $a=-\ln2/\ln\lambda$, with a constant $C_0$ not depending on $n$.
Let $s_k:=\rho \lambda^k\theta^n(1-\theta)$. For every $k\ge 1$ and every $s\in[s_{k+1},s_k]$, we observe that
\begin{eqnarray*}
\mu(B(x,r+s)\setminus B(x,r) )&\le&
\mu(B(x,r+s_k)\setminus B(x,r) )\\
&\le& m(I_{k-1})=2^{-k+1}m(I_0)\\
&\le& 2^{-k+1}\mu(B(x,\theta^n)\setminus B(x,\theta^{n+1}))\\
&\le& 2^{-k+1}\mu(B(x,2r)),\ \ \mbox{since}\ \ \theta^n<\frac r\theta<2r\\
&\le&4 s^a r^{-a}\left(\frac{\theta}{\rho(1-\theta)}\right)^a \mu(B(x,2r)),
\end{eqnarray*}
since $s^a\ge s_{k+1}^{a}=2^{-k-1}(\rho\theta^n(1-\theta))^a$ and $r^{-a}\ge (\theta^n)^{-a}\theta^{-a}$.
Now, for $s\in[s_1,r]$, we have $s/r\ge \rho\lambda(1-\theta)$ and so
$$\mu(B(x,r+s)\setminus B(x,r) )\le\mu(B(x,2r))\le s^a r^{-a} (\rho\lambda(1-\theta))^{-a}\mu(B(x,2r)).$$
\end{proof}

%%%%%%%%%%%%%%%

\end{document}